\documentclass[11pt,reqno]{article}
\usepackage[utf8]{inputenc}
\usepackage[a4paper, margin=2 cm]{geometry}
\usepackage{amsmath, amsthm, amssymb, amsfonts}
\usepackage{mathtools}
\usepackage{hyperref}
\usepackage{enumitem}
\usepackage{setspace}
\usepackage{xcolor}
\usepackage{thmtools}
\usepackage{thm-restate}
\usepackage{verbatim}
\usepackage{enumitem}

\usepackage{tikz}
\usetikzlibrary{arrows.meta}
\usepackage{graphicx}
\usepackage{caption}
\usepackage{subcaption}

\newtheorem{theorem}{Theorem}[section]
\newtheorem{lemma}[theorem]{Lemma}
\newtheorem{claim}[theorem]{Claim}
\newtheorem{definition}[theorem]{Definition}

\newtheorem{observation}[theorem]{Observation}

\newtheorem{question}[theorem]{Question}
\newtheorem{conjecture}[theorem]{Conjecture}

\title{Extending Thomassen's conjecture to directed graphs}
\author{Micha Christoph\thanks{Department of Mathematics, Institute for Operations Research, ETH Z\"{u}rich, Switzerland. \\ \textbf{\{micha.christoph,barnabas.janzer,raphaelmario.steiner\}@math.ethz.ch}. M.C. and R.S. were supported by the SNSF Ambizione Grant No. 216071. B.J. was supported by SNSF grant
200021-19696} \and Barnab\'{a}s Janzer\footnotemark[1]  \and Kalina Petrova\thanks{
Institute of Science and Technology Austria (ISTA), Austria. \textbf{kalina.petrova@ist.ac.at}. This project has received funding from the European Union’s Horizon 2020 research and innovation programme
under the Marie Skłodowska-Curie grant agreement No 101034413.\includegraphics[width=5.5mm, height=4mm]{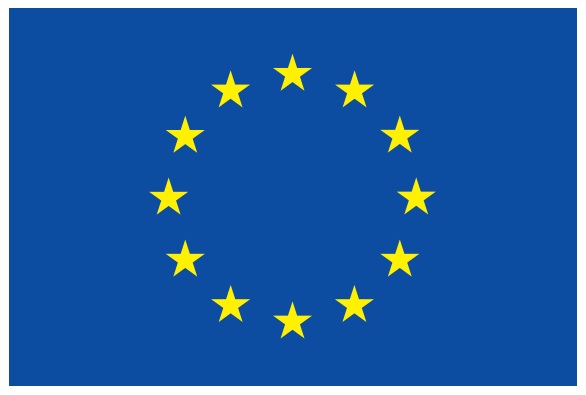}} \and Raphael Steiner\footnotemark[1]  
}
\begin{document}

\maketitle

\begin{abstract}
A famous conjecture by Thomassen from 1983 asserts that for any given $k,g\in \mathbb{N}$ there exists some $d=d(k,g)\in \mathbb{N}$ such that every graph of minimum degree at least $d$ contains a subgraph of minimum degree at least $k$ and girth at least $g$. In this paper, we initiate the systematic study of the directed analogs of Thomassen's conjecture one obtains when replacing minimum degree by minimum out-degree. Concretely, we study which digraphs $F$ are \emph{avoidable} in the sense that there exists $d_F:\mathbb{N}\rightarrow \mathbb{N}$ such that every digraph of minimum out-degree at least $d_F(k)$ contains an $F$-free subdigraph of minimum out-degree at least $k$. 
Among our main results, we show that all orientations of $C_3$ and $C_5$ are avoidable, while one-directed orientations of complete bipartite graphs and all oriented trees are not avoidable. This, in particular, shows that the most direct extension of Thomassen's conjecture to digraphs is false. 
We also fully characterize which digraphs are avoidable when restricting the setting to regular host digraphs. Finally, we raise numerous attractive open problems in the hope of sparking further progress.
\end{abstract}
\section{Introduction}
Thomassen's conjecture~\cite{THOMASSEN} from 1983 states that for all numbers $k$ and $g$ there exists $d=d(k,g)$ such that every graph of minimum degree at least $d$ contains a subgraph of minimum degree at least $k$ and girth at least $g$. Thomassen's conjecture was proved for $g=6$ by Kühn and Osthus~\cite{kuhn2004every}, see also~\cite{benny}. It remains open whether it holds for any $g\ge 7$. 

An equivalent formulation of the conjecture is as follows. Call a graph $F$ \textit{avoidable} if there exists $d_F:\mathbb{N}\rightarrow \mathbb{N}$ such that every graph of minimum degree at least $d_F(k)$ contains an $F$-free subgraph of minimum degree at least $k$. Thomassen's conjecture then states that all graphs containing a cycle are avoidable. If true, this would be a full characterization, as it is easy to see that no forest is avoidable.

In this paper, we initiate the study of the natural analog of Thomassen's conjecture for digraphs. Call a digraph $F$ \textit{avoidable} if there exists some $d_F:\mathbb{N}\rightarrow \mathbb{N}$ such that every digraph of minimum out-degree at least $d_F(k)$ contains an $F$-free subdigraph of minimum out-degree at least $k$. Which digraphs are avoidable? At first glance, this question may seem quite similar to the original conjecture by Thomassen. However, as the following discussion and our main results will demonstrate, for digraphs the situation changes drastically. 

A first key difference between the problems is the behavior of odd cycles: It is folklore (and can be verified by considering a max-cut) that every graph of minimum degree at least $2k-1$ contains a bipartite subgraph of minimum degree at least $k$. This immediately implies that all odd cycles are avoidable, and, thus, for Thomassen's conjecture it suffices to consider even cycles.

However, for digraphs such a simple reduction of the problem to bipartite graphs does not exist, as one can in general no longer pass to a bipartite subdigraph of large minimum out-degree. Indeed, as shown by Thomassen~\cite{thomassen2}, there exist digraphs $D$ with arbitrarily large minimum out-degree such that every directed cycle has odd length. It follows immediately that every bipartite subdigraph of $D$ has minimum out-degree~$0$. Hence, the directed version of Thomassen's problem already becomes non-trivial and interesting when considering orientations of odd cycles. As the first main result of this paper, we show that all orientations of small odd cycles are indeed avoidable.
\begin{theorem}\label{thm: main theorem}
All orientations of $C_3$ and $C_5$ are avoidable.
\end{theorem}
Recall that the \emph{odd-girth} of a graph is the length of its shortest odd cycle. Then, in particular, our arguments show that there exists a polynomially bounded function $f:\mathbb{N}\rightarrow \mathbb{N}$ such that every digraph of minimum out-degree at least $f(k)$ contains a subdigraph of minimum out-degree at least $k$ whose underlying graph has odd-girth at least $7$.

The special cases of the above result when we consider the \emph{directed} orientations of $C_3$ and $C_5$ were already previously known. Namely, Dellamonica, Koubek, Martin and R\"odl~\cite{directed_cycles} proved, using a clever application of the Lov\'{a}sz local lemma, that all directed cycles (of arbitrary length) are avoidable. Unfortunately, their proof technique is quite specifically suited to the directed orientations of cycles and does not generalize to any other, non-directed, orientations of cycles.
\begin{theorem}[Dellamonica, Koubek, Martin and R\"odl~\cite{directed_cycles}]\label{lem:remove_directed_cycles}
    Every directed cycle is avoidable.
\end{theorem}

As our second main result, we show the obvious extension of Thomassen's conjecture to digraphs is false: For every even-length cycle, its anti-directed orientation (with alternating directions of arcs) is not avoidable. This follows from the following stronger result, proving that even arbitrarily large complete bipartite graphs cannot be avoided. A \emph{one-directed complete bipartite graph} is an orientation of a complete bipartite graph all whose arcs are directed from one to the other color class.
\begin{theorem}\label{thm:complete directed pair}
No one-directed complete bipartite graph is avoidable.
\end{theorem} 

This exhibits an interesting difference between the undirected and directed settings, since, as mentioned, it is known that the undirected $C_4$ is avoidable, while by Theorem~\ref{thm:complete directed pair} its anti-directed orientation is not.

Next, let us consider the avoidability of digraphs without cycles, i.e., orientations of forests. In the undirected setting, a simple greedy embedding proves that every graph of minimum degree at least $k-1$ contains every forest on $k$ vertices as a subgraph, and hence, no forest is avoidable in undirected graphs. However, the same reasoning does not apply to oriented forests: It was recently shown by Hons et al.~\cite{hons2025unavoidablesubgraphsdigraphslarge} and the first and last authors~\cite{KAMAK} that only a rather restricted class of oriented forests, so-called \emph{grounded forests} (definition following further below) can be forced by high minimum out-degree. Despite this discrepancy, our next main result shows that no oriented forests are avoidable.\begin{theorem}\label{thm:treeavoidable}
No oriented forest is avoidable.
\end{theorem}

As a first step towards Thomassen's conjecture one might attempt to prove the statement for $d$-regular graphs instead of graphs with minimum degree $d$. This happens to be not so difficult and a suitable random sub-sampling of edges combined with Lov\'{a}sz' local lemma can be used to find spanning subgraphs of $d$-regular graphs with high minimum degree and girth. As a consequence, the only graphs which cannot be avoided in this setting are forests. One might impose a similar restriction for digraphs. A digraph $D$ is called \emph{$d$-regular} if every vertex has in- and out-degree $d$. We say that a digraph $F$ is \textit{regular-avoidable} if there exists $d_F:\mathbb{N}\rightarrow \mathbb{N}$ such that every $d_F(k)$-regular digraph contains an $F$-free subdigraph of minimum out-degree $k$. As it turns out, just like in the undirected case, it is much simpler to classify which digraphs are regular-avoidable.

To state the result, we need to give a precise definition of the aforementioned \emph{grounded forests}. Let us define a \emph{height function} of a digraph $D$ as any mapping $h:V(D)\rightarrow \mathbb{Z}$ such that $h(v) = h(u)+1$ for every arc $(u,v)$ of $D$. Note that every oriented forest admits a height function and that the latter is unique up to uniform shifts within connected components. We say that an oriented forest is a \emph{grounded forest} if it admits a height function that is constant on the set of vertices of in-degree at least $2$. Our last main result precisely characterizes regular-avoidable digraphs, as follows.
\begin{theorem}\label{thm: regular avoidable}
    A digraph is regular-avoidable if and only if it is not a grounded forest.
\end{theorem}
\section{Notation and preliminaries}
Throughout the paper, given a digraph $D$, we denote by $V(D)$ and $A(D)\subseteq \{(u,v)\in V(D)^2\mid u\neq v\}$ its set of vertices and arcs, respectively, and we consider an arc $(u,v)$ as starting in $u$ and ending at $v$, and call $u$ the \emph{tail} and $v$ the \emph{head} of the arc. The \emph{out-neighborhood} $N_D^+(v)$ and \emph{in-neighborhood} $N_D^-(v)$ are defined as the set of vertices to or from which $v$ has an arc. The out- and in-degree of a vertex $v$ are defined as $d_D^+(v):=|N_D^+(v)|$ and $d_D^-(v):=|N_D^-(v)|$, respectively. The minimum/maximum out- and in-degrees, denoted by $\delta^+(D), \delta^-(D), \Delta^+(D), \Delta^-(D)$ respectively, are defined accordingly. For $V\subseteq V(D)$ and an integer $i$, we write $N_i^+(V)$ (respectively $N_i^-(V)$) for the set of vertices to (respectively from) which there is a directed walk of length $i$ from (respectively to) $V$.

Throughout the paper, we work with digraphs of arbitrary size but bounded out-degree. In order to use random sub-sampling tricks while maintaining large minimum out-degree, we use the Lov\'asz Local Lemma.\begin{lemma}[The Lov\'asz Local Lemma~\cite{LLL,shearer1985problem}]
\label{lem:LLL}
Let $A_1,\ldots A_t$ be a sequence of events, each of which occurs with probability at most $q$ and is independent of all the other events, except at most $\Delta$ of them. Then, with positive probability none of the events occur if $eq\Delta<1$.
\end{lemma}
For the proof of Theorem~\ref{thm: main theorem}, it will be essential to use a digraph extension of the fact that every undirected graph admits a $2$-coloring of the vertices such that for each vertex at least half its neighbors have a different color from itself. As mentioned earlier, one cannot always reduce digraphs with large minimum out-degree to bipartite digraphs while maintaining large minimum out-degree. However, it is possible when we allow ourselves one more color, as was first established by Alon~\cite{noga}.
\begin{theorem}[\cite{noga}]\label{thm: majority coloring}
    Let $D$ be a digraph. Then there exists a $3$-coloring of $V(D)$ such that for each vertex $v$, at least a third of its out-neighbors have a different color from $v$. 
\end{theorem}
Kreutzer, Oum, Seymour, van der Zypen and Wood \cite{kreutzer} conjectured that the third in the above statement can be improved to a half. Besides the usual Chernoff bound, we also use the following version.
\begin{lemma}[\cite{chernoff}]\label{lem: chernoff}
    Let $X_1,\ldots,X_n$ be independent random variables taking values in $\{0,1\}$ and let $X=\sum_{i=1}^n X_i$. For $t\geq(2e-1)\mathbb{E}[X]$, we have
    $$
    \Pr[X>t]<2^{-t}.
    $$
\end{lemma}

\section{Constructions}
    In this section we prove Theorems~\ref{thm:complete directed pair} and \ref{thm:treeavoidable}. Both of these rely on $d$-out-arborescences. A $d$-out-arborescence of height $\ell$ is a rooted tree of height $\ell$ where each arc is directed away from the root and every vertex not at height $\ell$ has exactly $d$ children.
    \begin{proof}[Proof of Theorem~\ref{thm:complete directed pair}]
        Let us fix a one-directed complete bipartite graph $(A,B)$ and $k\geq |B|$. We show that for every $d\geq k$ there exists a digraph $D$ with minimum out-degree $d$ such that every $D'\subseteq D$ with minimum out-degree $k$ contains a copy of $(A,B)$. Let $T_1,\ldots T_d$ be $d$ disjoint $d$-out-arborescences of height $\ell=\left\lceil\log_k\left(|A|\binom{d}{k}\right)\right\rceil$. Let $D$ be the digraph consisting of these arborescences where we add an out-edge from every leaf to every root. Note that $D$ has minimum out-degree $d$. Let $D'$ be any subgraph of $D$ with minimum out-degree $k$. As every directed cycle in $D$ contains a root, $D'$ contains a root $r$ of $T_i$ for some $i$. Then, $D'$ contains at least $k^\ell\geq|A|\binom{d}{k}$ leaves of $T_i$. At least $|A|$ of them connect to the same $k$ roots in $D'$, yielding a copy of $(A,B)$.
    \end{proof}
    Before we prove Theorem~\ref{thm:treeavoidable}, we need a bit more notation. A digraph is \emph{rooted} if there is a special vertex $r$ (the \emph{root}) such that every vertex can be reached by a directed path starting from $r$. Given a rooted directed graph $D$, the \emph{$i$\textsuperscript{th} layer} $L_i(D)$ is the set of vertices at distance $i$ from the root. $D$ is \emph{layered} if all arcs of $D$ are of the form $(u,v)$ with $u\in L_i(D), v\in L_{i+1}(D)$ for some $i$, and, in addition, all vertices with out-degree zero belong to the same layer. We view the layer containing the out-degree-zero vertices as the ``bottom'' layer, and the zeroth layer $\{r\}$ as the ``top'' layer. The following lemma is the core of the proof of Theorem~\ref{thm:treeavoidable}.
    \begin{lemma}\label{lem:digraphtree}
        For all positive integers $k\geq 2$, $d$ and $t$, there exists a layered rooted digraph $D$ with out-degree $d$ at every vertex not in the bottom layer such that the following property holds. Whenever we take a subdigraph $H$ with out-degree at least $k$ at each vertex (apart from those in the bottom layer of $D$) such that the root is contained in $H$, then $H$ has an induced subdigraph (with non-empty vertex set) contained in the bottom $t$ layers in which all vertices have both in- and out-degree at least $k$, except the vertices in the first and $t$\textsuperscript{th} layer from the bottom, which have out- and in-degree zero, respectively.
    \end{lemma}
    \begin{proof}
        We proceed by induction on $t$. The case $t=1$ is trivial, as we can take $D$ to be a single vertex. Now, assume $t\geq 2$, and for $t-1$ we have already constructed such a digraph $F$.
	
        Take a $d$-out-arborescence $T_0$ of height $\ell=\left\lceil \log_k\left(2^d k\right) \right\rceil$, rooted at some vertex $r$. Let $L$ denote the set of leaves. For each $w\in L$, attach a copy $F_w$ of $F$ rooted at the vertex $w$. Add an additional set $B$ of $d$ vertices, and for all $w\in L$, add an edge from each vertex in the bottom layer of $F_w$ to each vertex in $B$. Let the resulting digraph be $D$. It is easy to see that it is rooted (at $r$), layered, and every vertex not in the bottom layer $B$ has out-degree $d$.
	
        Now assume that $H$ is a subdigraph of $D$ with $r\in V(H)$ such that each vertex in $V(H)\setminus B$ has out-degree at least $k$ in $H$. It follows that $|L\cap V(H)|\geq k^\ell$. For each $w\in L\cap V(H)$, by induction, we know that $H[V(F_w)]$ has an induced subdigraph $H'_w$ (with non-empty vertex set) contained in the bottom $t-1$ layers of $F_w$ in which all vertices have both in- and out-degree at least $k$, except the vertices in the first and $(t-1)$\textsuperscript{st} layer from the bottom, which have out- and in-degree zero, respectively.
	
        For each $w\in L\cap V(H)$, let $B_w$ denote the set of vertices in $B$ which have an in-neighbour in $V(H'_w)$ in $H$. Since $|L\cap V(H)|\geq k^\ell\geq 2^{d}k$, there exist $w_1,\dots,w_k\in L\cap V(H)$ distinct such that $B_{w_i}$ is the same set $B^*\subseteq B$ for all $i\in [k]$. Then $H[\bigcup_{i=1}^k V(H_{w_i})\cup B^*]$ satisfies the conditions.
\end{proof}
\begin{proof}[Proof of Theorem~\ref{thm:treeavoidable}]
        Given a directed tree $T$, we want to show that there is a positive integer $k$ such that for all $d$, there is a digraph with minimum out-degree $d$ such that all subgraphs with minimum out-degree $k$ contain $T$. Let $k=|V(T)|$, and given $d$, apply Lemma~\ref{lem:digraphtree} with parameters $k$, $d$ and $t=2|V(T)|$, to obtain a digraph $D$. Take $d$ disjoint copies $D_1,\dots,D_d$ of $D$, and take an edge from every bottom vertex to every root of some $D_i$. Then the out-degree of every vertex is $d$, and every subgraph $H$ with minimum out-degree $k$ contains the root of $D_i$ for some $i$. Using Lemma~\ref{lem:digraphtree}, the bottom $t$ layers of $H[V(D_i)\cap V(H)]$ contain an induced subgraph $H'$ where all vertices have both in- and out-degree at least $k$, apart from those in the first and the $t$\textsuperscript{th} layer of $D_i$ from the bottom. Then, we can greedily embed $T$ into this subgraph $H'$, starting with embedding an arbitrary vertex into the $|V(T)|$\textsuperscript{th} layer from the bottom. 
\end{proof}

\section{Proof of Theorem~\ref{thm: main theorem}}
	\begin{figure}[b]
		\centering
		\begin{subfigure}[h]{0.24\textwidth}
			\centering
			\begin{tikzpicture}[scale=2,
    node_style/.style={circle, draw, fill=black, minimum size=4pt, inner sep=0pt},
    edge_style/.style={->,thick,>={Stealth[length=2mm, width=2mm]}}
]
    \node[node_style] (v1) at (90:.5cm)  {};
    \node[node_style] (v2) at (18:.5cm)  {};
    \node[node_style] (v3) at (-54:.5cm) {};
    \node[node_style] (v4) at (-126:.5cm){};
    \node[node_style] (v5) at (-198:.5cm){};

    \draw[edge_style] (v1) -> (v2);
    \draw[edge_style] (v2) -> (v3);
    \draw[edge_style] (v3) -> (v4);
    \draw[edge_style] (v4) -> (v5);
    \draw[edge_style] (v5) -> (v1);
\end{tikzpicture}
			\captionsetup{justification=centering}
			\caption*{$C_5^{(1)}$}
		\end{subfigure}
		\begin{subfigure}[h]{0.24\textwidth}
			\centering
\begin{tikzpicture}[scale=2,
    node_style/.style={circle, draw, fill=black, minimum size=4pt, inner sep=0pt},
    edge_style/.style={->,thick,>={Stealth[length=2mm, width=2mm]}}
]
    \node[node_style] (v1) at (90:.5cm)  {};
    \node[node_style] (v2) at (18:.5cm)  {};
    \node[node_style] (v3) at (-54:.5cm) {};
    \node[node_style] (v4) at (-126:.5cm){};
    \node[node_style] (v5) at (-198:.5cm){};

    \draw[edge_style] (v1) -> (v2);
    \draw[edge_style] (v2) -> (v3);
    \draw[edge_style] (v3) -> (v4);
    \draw[edge_style] (v4) -> (v5);
    \draw[edge_style] (v1) -> (v5);
\end{tikzpicture}			\captionsetup{justification=centering}
			\caption*{$C_5^{(2)}$}
		\end{subfigure}		\begin{subfigure}[h]{0.24\textwidth}
			\centering
\begin{tikzpicture}[scale=2,
    node_style/.style={circle, draw, fill=black, minimum size=4pt, inner sep=0pt},
    edge_style/.style={->,thick,>={Stealth[length=2mm, width=2mm]}}
]
    \node[node_style] (v1) at (90:.5cm)  {};
    \node[node_style] (v2) at (18:.5cm)  {};
    \node[node_style] (v3) at (-54:.5cm) {};
    \node[node_style] (v4) at (-126:.5cm){};
    \node[node_style] (v5) at (-198:.5cm){};

    \draw[edge_style] (v1) -> (v2);
    \draw[edge_style] (v2) -> (v3);
    \draw[edge_style] (v3) -> (v4);
    \draw[edge_style] (v5) -> (v4);
    \draw[edge_style] (v1) -> (v5);
\end{tikzpicture}
			\captionsetup{justification=centering}
			\caption*{$C_5^{(3)}$}
		\end{subfigure}		\begin{subfigure}[h]{0.24\textwidth}
			\centering
\begin{tikzpicture}[scale=2,
    node_style/.style={circle, draw, fill=black, minimum size=4pt, inner sep=0pt},
    edge_style/.style={->,thick,>={Stealth[length=2mm, width=2mm]}}
]
    \node[node_style] (v1) at (90:.5cm)  {};
    \node[node_style] (v2) at (18:.5cm)  {};
    \node[node_style] (v3) at (-54:.5cm) {};
    \node[node_style] (v4) at (-126:.5cm){};
    \node[node_style] (v5) at (-198:.5cm){};

    \draw[edge_style] (v1) -> (v2);
    \draw[edge_style] (v2) -> (v3);
    \draw[edge_style] (v4) -> (v3);
    \draw[edge_style] (v4) -> (v5);
    \draw[edge_style] (v1) -> (v5);
\end{tikzpicture}
			\captionsetup{justification=centering}
			\caption*{$C_5^{(4)}$}
		\end{subfigure}
                \caption{The four different orientations of $C_5$}
        \label{Fig: C5}

	\end{figure}
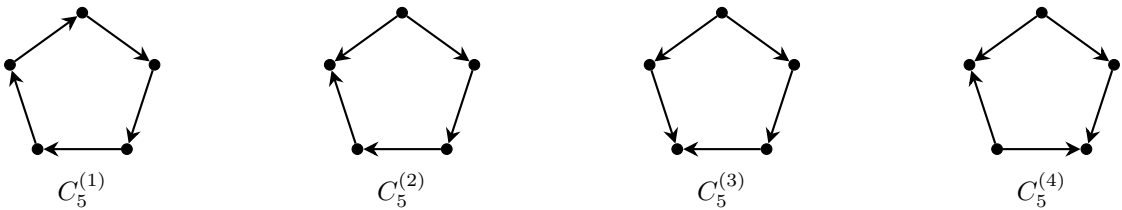
In this section we show that every orientation of $C_3$ and $C_5$ is avoidable. Let $C_3^{(1)}$ be the directed and $C_3^{(2)}$ the transitive triangle. Figure~\ref{Fig: C5} shows the $4$ different orientations of $C_5$. Theorem~\ref{thm: majority coloring} allows us to assume that $D$ is $3$-partite, where we say that a digraph $D$ is $3$-partite with partition $(A,B,C)$ if the underlying undirected graph is $3$-partite with parts $A,B$, and $C$. The following definition (and the lemma thereafter) asserts even more control on the neighborhoods of each vertex.
\begin{definition}[$s$-typed]
    Given a non-negative integer $s$ and a $3$-partite digraph $D$ with partition $(A,B,C)$, we say that $D$ is $s$-typed, if for every $v\in V(D)$ there exists a word $t\in\{A,B,C\}^s$ such that $N^+_i(v)\subseteq t(i)$ for every $1\le i \le s$. We call $t$ the $s$-type of $v$ (or just type if $s$ is clear from the context). Furthermore, we say that the partition $(A,B,C)$ witnesses that $D$ is $s$-typed.
\end{definition}
\begin{lemma}\label{lem: typing}
    Every $3$-partite digraph $D$ with minimum out-degree $d$ contains an $s$-typed subgraph $D'$ with minimum out-degree at least $d/3^{s}$.
\end{lemma}
\begin{proof}
    We proceed by induction on $s$. Note that every $3$-partite digraph is $0$-typed. Suppose $s\geq 1$ and the statement holds for $s-1$. Then, it follows from the induction hypothesis that there exists $(s-1)$-typed $F\subseteq D$ with minimum out-degree at least $d/3^{s-1}$. Let $(A,B,C)$ be a partition of $V(F)$ which witnesses that $F$ is $(s-1)$-typed.
    
    For each $v\in V(F)$, denote by $t_v$ the $(s-1)$-type of $v$. (If $s=1$, then instead we take $t_v$ to be the set from $\{A,B,C\}$ containing $v$.) Since $F$ is $(s-1)$-typed, it follows that $t_v(i)=t_u(i-1)$ for every $u\in N_F^+(v)$ and $2\leq i\leq s-1$. Note that this determines $t_u$ besides $t_u(s-1)$. Therefore, at least a third of the vertices in $N_F^+(v)$ have the same type $t_v'$. Let $D'\subseteq F$ be the digraph obtained by keeping only the out-edges between $v$ and the vertices with type $t_v'$ for every $v\in V(F)$. Then, $\delta^+(D')\geq d/3^{s}$, and for all $v\in V(D')$, $N^+(v)\subseteq t_v(1)$ and $N_i^+(v)\subseteq t_v'(i-1)$ for $2\leq i\leq s$. (If $s=1$, then the latter conditions are replaced by $N^+(v)\subseteq t'_v$.)
\end{proof}
Though we claimed earlier that Theorem~\ref{thm: majority coloring} is the directed extension of the reduction to the bipartite case in undirected graphs, it really is the combination of Theorem~\ref{thm: majority coloring} together with Lemma~\ref{lem: typing}. In bipartite graphs, the $i$\textsuperscript{th} neighborhood of a vertex is monochromatic and the color depends only on the parity. While we cannot achieve a property as strong as the parity condition, Lemma~\ref{lem: typing} at least allows us to assume that the $i$\textsuperscript{th} neighborhood is monochromatic for bounded $i$. As the following lemma shows, this is sufficient to show that $C_3^{(2)}$, $C_5^{(3)}$ and $C_5^{(4)}$ are avoidable.
\begin{lemma}\label{lem: easy orientations}
    Let $D$ be a $3$-partite $2$-typed digraph. Then, $D$ does not contain $C_3^{(2)}$, $C_5^{(3)}$ or $C_5^{(4)}$.
\end{lemma}
\begin{proof}
    Let $(A,B,C)$ be a partition of $V(D)$ witnessing that $D$ is $2$-typed. We say that a set is monochromatic if it is a subset of one of $A$, $B$ or $C$ and note that every monochromatic set is independent. Since $D$ is $2$-typed, $N^+(v)$ and $N_2^+(v)$ are monochromatic for every $v\in V(D)$, immediately excluding any copies of $C_3^{(2)}$ and $C_5^{(3)}$ from $D$. Suppose that $v$ and $u$ share an out-neighbor in $D$. Again since $D$ is $2$-typed, it follows that $N^+(v)\cup N^+(u)$ is monochromatic. When applied with $u$ and $v$ as the sources of $C_5^{(4)}$, this implies that $D$ does not contain any copy of $C_5^{(4)}$.
\end{proof}
Together with Theorem~\ref{lem:remove_directed_cycles}, it only remains to show that $C_5^{(2)}$ is avoidable. We start with the following simple observation. We say a digraph $D$ is $k$-degenerate if the underlying undirected graph is $k$-degenerate. 
\begin{observation}\label{obs: degenerate}
    Let $D$ be a (multi-)digraph with $\Delta^+(D)\leq k$. Then, $D$ is $2k$-degenerate.
\end{observation}
The following lemma in essence proves that $C_5^{(2)}$ is avoidable as we will iteratively apply it to show that no vertex is the source of a copy of $C_5^{(2)}$.
\begin{lemma}\label{lem: C5}
    Let $k\geq 100$ and let $D$ be a $3$-partite digraph with minimum out-degree $k^{20}$. Suppose further that $D$ is $1$-typed with witness $(A,B,C)$ and $D$ contains no copy of $C_3^{(1)}$. Then, $D$ contains a spanning subgraph $D'$ such that each vertex has out-degree $k$ and $D'$ does not contain a copy of $C_5^{(2)}$ with source in $N_1^-(A)$.
\end{lemma}
\begin{proof}
    Let $V=N_1^-(A)$. Since $D$ is $1$-typed, the out-neighborhood of every vertex in $V$ is contained in $A$. As $A$ is an independent set, $V$ is also an independent set. Additionally, every vertex which shares an out-neighbor with a vertex in $V$ must itself be in $V$. It follows that every copy of $C_5^{(2)}$ with source in $V$ consists of a path of length $3$ from $u$ to $v$ and a common neighbor of $u$ and $v$ in $A$, where $u,v\in V$. Therefore, it seems natural to define the following auxiliary digraph. For $D'\subseteq D$, let $H(D')$ be an auxiliary (multi-)digraph on $V$, where, for every two vertices $u,v \in V$ with $N^+_{D'}(u) \cap N^+_{D'}(v) \neq \varnothing$ and every arc $(u, w) \in A(D')$ such that there is a path of length $2$ from $w$ to $v$ in $D'$, we add an arc $e = (u,v)$ to $A(H(D'))$. We say that $e$ corresponds to $w$. Note that this potentially results in parallel edges in case there is more than one such edge $(u,w)$ for the same pair $(u,v)$. 
    
    For every copy of $C_5^{(2)}$ in $D$ with source in $V$, there exists an arc $(u,v)\in H(D)$ with corresponding vertex $w\in A$ such that $u$ is the source and $w$ the first vertex on the directed path from $u$ to $v$ in this copy of $C_5^{(2)}$ (however one arc in $H(D)$ may correspond to several different copies of $C_5^{(2)}$). Therefore, if $H(D)$ contains no arcs then there are no copies of $C_5^{(2)}$ with source in $V$ in $D$. We will show that there exists an appropriate $D'\subseteq D$ such that $H(D')$ has no arcs. To do so, we proceed in two stages. First, we find a subgraph $F\subseteq D$ together with an ordering of the vertices in $V$ which limits the interaction between the out-edges of each vertex with the out-edges of the vertices preceding it. Then, we obtain $D'$ by processing the vertices in the given order, always restricting their neighborhood to a subset such that there is no copy of $C_5^{(2)}$ with any previous vertex.
    \begin{claim}\label{clm: main claim}
        There exists an ordering $v_1,\ldots,v_n$ of $V$ and a spanning subdigraph $F\subseteq D$ such that $d_F^+(v)\ge 3k^4$ for every $v\in V\cap V(F)$ and $d_F^+(v)=k$ for every $v\in V(F)\setminus V$. In addition, every vertex $v_i$ has at most $2k$ in-neighbors $v_j$ in $H(F)$ with $j<i$.
    \end{claim}
    Before we prove Claim~\ref{clm: main claim}, let us show how to make use of it. Let $D'\subseteq F$ be the digraph obtained by handling the vertices of $V$ in the given order, restricting their out-neighborhood to a subset of size $k$ one at a time the following way. Fix some $i$ and suppose we have already chosen the out-neighbors of $v_j$ for $j<i$. We would like to pick $N\subseteq N^+_{F}(v_i)$ of size $k$ such that if we connect $v_i$ to $N$ in $D'$ then there are no arcs in $H(D')[\{v_1,\ldots,v_i\}]$. Note that $H(D')\subseteq H(F)$ so it suffices to select $N$ such that none of the arcs in $H(F)[\{v_1,\ldots,v_i\}]$ remain. By the choice of $F$, there are at most $2k$ in-neighbors $v_j$ of $v_i$ with $j<i$ in $H(F)$. Since each such $v_j$ has $k$ out-neighbors in $D'$ (which have already been selected), there exists $N'\subseteq N^+_{F}(v_i)$ of size $3k^4-2k^2$ which is disjoint from the out-neighborhood of all these $v_j$. Restricting our attention to $N'$ guarantees us that $v_i$ will not have any in-neighbors in $H(D')[\{v_1,\ldots,v_i\}]$. Next, we consider arcs $(v_i,v_j)$, $j<i$, in $H(F)[\{v_1,\ldots,v_i\}]$. Recall that the arc $(v_i,v_j)$ corresponds to some vertex $w\in N^+_{F}(v_i)$. To exclude the copies of $C^{(2)}_5$ associated with $(v_i,v_j)$ and $w$, it is enough to exclude either $w$ or $N^+_{D'}(v_j)$ from $N$. Thus, it is sufficient to choose $N\subseteq N'$ in such a way that whenever $w_1,w_2\in N$, then there is no path of length $3$ from $w_1$ to $w_2$ through some vertex $v_j$ with $j<i$. So, let $G$ be the auxiliary undirected graph on $N'$ where $w_1$ and $w_2$ are connected exactly if such a path exists. Since every vertex $w$ is the starting point of at most $k^3$ such paths of length $3$, the average degree of $G$ is at most $2k^3$. Thus, $G$ contains an independent set $N$ of size $k$, since
    $$\frac{|N'|}{2k^3+1}=\frac{3k^4-2k^2}{2k^3+1}\geq k.$$

    After repeating the above for every $i$, we obtain $D'\subseteq F$ such that every vertex has out-degree $k$ and $H(D')$ is empty. By the definition of $H$, it follows that $D'$ does not contain any copies of $C_5^{(2)}$ with source in $V$.
    \begin{proof}[Proof of Claim~\ref{clm: main claim}]
        Set $d:=k^{20}$ and recall that $k\geq 100$. First, arbitrarily remove all but $k$ out-edges for every vertex in $V(D)\setminus V$, and all but $d$ out-edges for every vertex in $V$, to obtain $F_0\subseteq D$.
        Each vertex in $V$ can be the start of at most $dk^2$ paths of length $3$ in $F_0$ ending at some vertex in $V$. Therefore, every vertex has at most $dk^2$ out-edges in $H(F_0)$. Hence, by Observation~\ref{obs: degenerate}, $H(F_0)$ is $2dk^2$-degenerate. Let $v_1, \dots, v_n$ be an ordering of $V=V(H(F_0))$ witnessing the degeneracy. In particular, for each $v_i$, there are at most $2dk^2$ arcs $(v_j, v_i)$ with $j<i$. Let $F\subseteq F_0$ be obtained by sub-sampling every arc from $V$ to $A$ independently with probability $p=k^{-15}$. 
        
        Towards using Lemma~\ref{lem:LLL}, let us define the following events. For each vertex $v_i$, let $A_i$ be the event that $d^+_{F}(v_i) \notin [\frac{1}{2} dp, \frac{3}{2}dp]$ and let $B_i$ be the event that $v_i$ has more than $2k$ in-neighbors $v_j$ in $H(F)$ with $j<i$. Note that $dp/2\geq 3k^4$, so it is enough to show that with positive probability none of the events $A_i$ or $B_i$ happen.
        
       Besides the dependence of $A_i$ and $B_i$, $A_i$ only depends on events $B_j$ with $i<j$ and $(v_i,v_j)\in H(F_0)$, while $B_i$ also depends on events $A_j,B_j$ for $j<i$ and $(v_j,v_i)\in H(F_0)$, as well as events $B_\ell$ for which there is some $j<i,\ell$ with $(v_j,v_i),(v_j,v_\ell)\in H(F_0)$. As every vertex $v_i$ has at most $dk^2$ out-edges in $H(F_0)$ and by the choice of $v_1, \dots, v_n$, at most $2dk^2$ in-neighbors $v_j$ with $j<i$, it follows that the event $A_i\vee B_i$ depends on at most $5d^2k^4$ other events $A_j\vee B_j$. In the terminology of Lemma~\ref{lem:LLL}, we then have $\Delta\leq 5d^2k^4$.
        
       By a standard application of the Chernoff bound, we get that $\Pr[A_i] \leq 2e^{-dp/12}\leq e^{-k}$. Next, let us consider the event that $B_i$ happens conditioned on $\overline{A_i}$. For each $j <i$, let $X_{i,j}$ be the indicator variable that $v_j$ is an in-neighbor of $v_i$ in $H(F)$ and let $X_i =\sum_{j=1}^{i-1}X_{i,j}$. Note that $B_i$ is the event that $X_i>2k$. Let us reveal all the out-edges of $v_i$ in $F$ and note that $d^+_{F}(v_i)\leq 3dp/2$, as we condition on $\overline{A_i}$. Then, $(v_j, v_i) \in H(F)$ only if $N^+_{F}(v_j)$ intersects $N^+_{F}(v_i)$ and there is a path of length $3$ from $v_j$ to $v_i$ in $F$. Let $Y_{i,j}$ be the event that $N^+_{F}(v_j)$ intersects $N^+_{F}(v_i)$. By a union bound, we get 
    $$
    \Pr[Y_{i,j}]\leq \frac{3}{2}dp\cdot p.
    $$
    Suppose there are $m_{i,j}$ parallel arcs from $v_j$ to $v_i$ in $H(F_0)$. Each of these corresponds to a different $w\in A$ for which there is a path of length $3$ from $v_j$ to $v_i$ in $F_0$ through $w$. So, a path of length $3$ from $v_j$ to $v_i$ in $F$ exists only if at least one of the arcs $(v_j, w)$ is preserved in $F$. Call this event $P_{i,j}$. A union bound yields
    $$\Pr [P_{i,j}] \leq p m_{i,j}.$$
    Observe that the events $P_{i,j}$ and $Y_{i,j}$ are independent as $N^+_{F_0}(v_i)$ is disjoint from $$\{w\mid \text{there exists $(v_j,v_i)\in H(F_0)$ corresponding to $w$}\},$$ because $F_0$ does not contain a copy of $C_3^{(1)}$. Thus, we can conclude
    $$ \Pr[X_{i,j}=1] = \Pr[(v_j,v_i) \in H(F)]\le \Pr[Y_{i,j} \wedge P_{i,j}] = \Pr[Y_{i,j}] \Pr [P_{i,j}] \leq 3 dp^3 m_{i,j}$$
    and so, using $\sum_{j=1}^{i-1}m_{i,j}\leq 2dk^2$,
    $$ \mathbb{E}[X_i] = \sum_{j=1}^{i-1} \Pr[X_{i,j}=1] \leq \sum_{j=1}^{i-1} 3dp^3 m_{i,j} \leq 6d^2k^2p^3\leq 1.$$
    Recall that we already revealed the out-edges from $v_i$. Consequently, the variables $X_{i,1}, \dots, X_{i,i-1}$ are independent since $X_{i,j}$ only depends on the out-edges from $v_j$. Thus, we can apply Lemma~\ref{lem: chernoff} to $X_i$ to deduce that
    $$ \Pr[X_i > 2k] \leq 2^{-2k}.$$
    As this holds no matter the arcs we revealed from $v_i$, we get 
    $$\Pr[B_i\mid \overline{A}_i]\leq 2^{-2k}.$$
    Finally, we get
    $$
    \Pr[A_i\vee B_i] \leq \Pr[A_i]+\Pr[B_i\mid \overline{A_i}]\leq e^{-k}+2^{-2k}<\frac{1}{e\Delta}.
    $$
    By Lemma~\ref{lem:LLL}, with positive probability, none of the events $A_i$ or $B_i$ happen.
    \end{proof}
\end{proof}
\begin{proof}[Proof of Theorem~\ref{thm: main theorem}]
Instead of showing that each of the orientations of $C_3$ and $C_5$ is avoidable, we show that we can avoid them all at the same time, which certainly implies that each of them is avoidable. Without loss of generality, let us assume that $k\geq 100$. Let $d$ be large enough such that, by Theorem~\ref{lem:remove_directed_cycles}, there exists $D_1\subseteq D$ with minimum out-degree $27k^{20^3}$ such that $D_1$ does not contain a copy of $C_3^{(1)}$ or $C_5^{(1)}$.
Then, apply Theorem~\ref{thm: majority coloring} to $D_1$ to obtain $3$-partite $D_2\subseteq D_1$ with minimum out-degree $9k^{20^3}$. Furthermore, by Lemma~\ref{lem: typing}, there exists a $2$-typed $D_3\subseteq D_2$ with minimum out-degree $k^{20^3}$. By Lemma~\ref{lem: easy orientations}, $D_3$ does not contain a copy of $C_3^{(2)}$, $C_5^{(3)}$ or $C_5^{(4)}$. Therefore, $D_3$ does not contain a triangle and every cycle of length $5$ in $D_3$ is a copy of $C_5^{(2)}$. 

Let $(A,B,C)$ be a partition of $V(D_3)$ witnessing that $D_3$ is $2$-typed. By Lemma~\ref{lem: C5}, there exists $D_4\subseteq D_3$ with minimum out-degree $k^{20^2}$ such that no vertex in $N^-_{D_4}(A)$ is the source of a copy of $C_5^{(2)}$ in $D_4$. Similarly, we may get $D_5\subseteq D_4$ with minimum out-degree $k^{20}$ such that no vertex in $N^-_{D_5}(B)$ is the source of a copy of $C_5^{(2)}$ in $D_5$ and, finally, $D_6\subseteq D_5$ with minimum out-degree $k$ such that no vertex in $N^-_{D_6}(C)$ is the source of a copy of $C_5^{(2)}$ in $D_6$. Then, $D_6$ does not contain a copy of $C_5^{(2)}$, as $(N^-_{D_4}(A),N^-_{D_5}(B),N^-_{D_6}(C))$ is a partition of $V(D_6)$.
\end{proof}

\section{Regular-avoidable digraphs}
In this section we prove Theorem~\ref{thm: regular avoidable}. As mentioned before, it was proved~\cite{KAMAK} by the first and fourth authors of this paper that any grounded tree is contained in all digraphs with sufficiently large minimum out-degree. Hence, grounded forests are not regular-avoidable. So, it remains to show that if $H$ is not a grounded forest then it is regular-avoidable. Our proof distinguishes two cases, namely whether $H$ is a forest or not. The following lemma resolves the case when $H$ is not a forest.
\begin{lemma}\label{lem: regular cycles}
    Let $k,\ell\in \mathbb{N}$ such that $k$ is large enough as a function of $\ell$. Let $D$ be a $2k^{\ell+1}$-regular digraph. Then, there exists $D'\subseteq D$ with minimum out-degree $k$ and no cycle of length $\ell$.
\end{lemma}
\begin{proof}
    Towards applying Lemma~\ref{lem:LLL}, set $d=2k^{\ell+1}$, $p=k^{-\ell}$ and let $D'\subseteq D$ be obtained by sub-sampling each arc independently with probability $p$. For each (not necessarily directed) cycle $C$ of length $\ell$ in $D$, denote by $A_C$ the event that $C$ is in $D'$. For each $v\in V(D)$, denote by $B_v$ the event that $d^+_{D'}(v)\leq dp/2= k$.
    
    Note that each edge is contained in at most $(2d)^{\ell-2}$ cycles in $D$. So, $A_C$ depends on at most $\ell(2d)^{\ell-2}$ other events $A_{C'}$ and $B_v$ depends on at most $d\cdot (2d)^{\ell-2}$ events $A_{C'}$. Furthermore, $A_C$ also depends on $\ell$ events $B_v$. Therefore, none of the above events depends on more than $\Delta=\ell(2d)^{\ell-1}\ll k^{\ell^2}$ other events. One can see that $\Pr[A_C]=p^{\ell}=k^{-\ell^2}$ and, by the Chernoff bound, we get 
    $$
    \Pr[B_v]\leq e^{-dp/12}=e^{-k/6}\ll k^{-\ell^2}.
    $$
    By Lemma~\ref{lem:LLL}, there exists $D'\subseteq D$ such that none of the events $A_C$ or $B_v$ occur. Then, $D'$ has minimum out-degree at least $k$ and no cycle of length $\ell$.
\end{proof}
If $H$ is a forest but not grounded then there are two vertices in $H$ with in-degree more than $1$ such that there is a path between them having a different number of forward and backward arcs. The following lemma allows us to pass to a subdigraph in which every such path is long.
\begin{lemma}\label{lem: partitioning}
    Let $t$ be a positive integer, $k$ a sufficiently large integer as a function of $t$, and $D$ a $k^{2t}$-regular digraph. Then, there exists $D'\subseteq D$ with minimum out-degree $k$ and partition $(V_1,\ldots,V_t)$ of $V(D')$ such that every arc in $D'$ goes from $V_i$ to $V_{i+1}$ for some $1\leq i\leq t$, (where $V_{t+1}=V_1$), and the vertices in $V(D')\setminus V_1$ have in-degree at most $1$. 
\end{lemma}
\begin{proof}
    Set $d=k^{2t}$, $p_i = \frac{1-(6k)^{-1}}{1-(6k)^{-t}}\cdot(6k)^{i-t}$ for $1\leq i\leq t$. While the choice of these $p_i$ is mostly arbitrary, note that $\sum p_i=1$ and we also use the following two inequalities. First, it will be handy that $p_i d\gg k$ for all $i$ and, second, $p_{i+1}/p_{i}= 6k$ for all $1\leq i\leq t-1$.
    \begin{claim}
        There exists a partition $(V_1,\ldots,V_{t})$ of $V(D)$ such that for all $v\in V(D)$ and $1\leq i\leq t$, it holds that $d^+(v,V_i)\geq p_id/2$ and $d^-(v,V_i)\leq 3p_id/2$.
    \end{claim}
    \begin{proof}
        Assign to each vertex a color from $\{1,\ldots,t\}$ independently of the other vertices, where we assign color $i$ with probability $p_i$. With the help of Lemma~\ref{lem:LLL}, we show that this coloring gives the desired partition with positive probability, where $V_i$ is the set of vertices assigned color $i$. For each vertex $v$, let $A_{v}$ denote the event that there exists $1\leq i\leq t$ such that $v$ has less than $p_id/2$ out-neighbors with color $i$ or more than $3p_id/2$ in-neighbors with color $i$. Note that each event $A_v$ only depends on the coloring of the vertices adjacent to $v$. As $D$ is $d$-regular it follows that $A_v$ depends on at most $(2d)^2$ other events $A_w$. Therefore, it remains to show $\Pr[A_v]<1/(4ed^2)$. Let us fix a vertex $v\in V(D)$ and some color $1\leq i\leq t$. Let $X_{v,i}^+, X_{v,i}^-$ denote the number of out- (respectively in-)neighbors of $v$ with color $i$. As $X_{v,i}^+$ follows a binomial distribution with mean $p_id$, Chernoff's bound implies $\Pr[X_{v,i}^+<p_id/2]\leq e^{-p_id/12}\leq e^{-k}$. Similarly, we get $\Pr[X_{v,i}^->3p_id/2]\leq e^{-p_id/8}\leq e^{-k}$. By a union bound over all colors, we get
        $$
        \Pr[A_v]\leq 2te^{-k}\ll d^{-2}.
        $$
    \end{proof}
    Let $F\subseteq D$ be obtained by keeping only the arcs from $V_i$ to $V_{i+1}$ for each $i$, where $V_{t+1}=V_1$.
        Let $D'\subseteq F$ be obtained by sub-sampling a unique in-arc (if there exists one) for every $v\in V(F)\setminus V_1$ (while keeping all the in-arcs for $v\in V_1$). We use Lemma~\ref{lem:LLL} to show $D'$ satisfies the requirements with positive probability. Every vertex in $V(F)\setminus V_1$ is guaranteed to have in-degree at most $1$ by the choice of $D'$. Therefore, it suffices to show that every vertex $v\in V(F)$ has out-degree at least $k$. Let $A_v$ be the event that $v$ has out-degree less than $k$ in $D'$. If $v\in V_t$ then $d^+_{D'}(v)=d^+_F(v)\geq p_1d/2\geq k$. Suppose then that $v\in V_i$ for some $1\leq i\leq t-1$. For arcs $(v,u)\in A(F)$, denote by $X_{v,u}$ the indicator random variable for $u$ sub-sampling $(v,u)$ and set $X_v:=\sum_{(v,u)\in A(F)}X_{v,u}$. Then, $A_v$ is precisely the event that $X_v<k$. Observe that 
        $$\mathbb{E}[X_v]=\sum_{(v,u)\in A(F)}\mathbb{E}[X_{v,u}]=\sum_{(v,u)\in A(F)}\frac{1}{d^-_F(u)}\geq \frac{p_{i+1}d/2}{3p_{i}d/2} = 2k.$$ 
        A simple application of the Chernoff bound shows that $\Pr[A_v]=\Pr[X_v<k]\leq e^{-k/6}$.
        Note that each event $A_v$ only depends on events $A_w$, where $w$ shares an out-neighbor with $v$. As $F\subseteq D$ and $D$ is $d$-regular, it follows that $A_v$ depends on at most $d^2$ other events. As $d^2\cdot e^{-k/6}\ll 1$, the statement follows from Lemma~\ref{lem:LLL}.
\end{proof}
\begin{proof}[Proof of Theorem~\ref{thm: regular avoidable}]
    Let $H$ be any digraph which is not a grounded forest and let us assume without loss of generality that $k$ is large enough compared to $|V(H)|$. By Lemma~\ref{lem: regular cycles}, we may assume that $H$ is a forest. Let $D'\subseteq D$ be as in Lemma~\ref{lem: partitioning} with $t=|V(H)|$. It remains to show that $D'$ does not contain a copy of $H$. As $H$ is not grounded, there exist vertices $u,v\in V(H)$ with in-degree more than $1$ such that the path from $u$ to $v$ in $H$ has a different number of forward and backward arcs. Observe that this path has length at most $|V(H)|-1<t$. Suppose towards a contradiction that there exists an embedding of $H$ into $D'$. Then, $u$ and $v$ both get embedded into $V_1$ as they have in-degree more than $1$. But this is a contradiction, as every path starting and ending in $V_1$ of length less than $t$ has the same number of forward and backward edges.
\end{proof}
\section{Open questions}
The main open problem left from our work is to characterize which digraphs are avoidable. Given that this is closely linked (at least in spirit) to Thomassen's conjecture, which currently remains wide open for undirected graphs, a full characterization may be elusive at this point. Thus, in the following we highlight some interesting and challenging questions that we deem more approachable.
First of all, in light of Theorem~\ref{thm: main theorem}, we are tempted to conjecture that all orientations of odd cycles are avoidable.
\begin{conjecture}\label{conj}
    All orientations of odd cycles are avoidable.
\end{conjecture}
A positive resolution of Conjecture~\ref{conj} would establish a very nice analog to the existence of bipartite subgraphs of high minimum degree: It would yield that every digraph of large minimum out-degree has a subdigraph with still large minimum out-degree that has large odd girth and thus locally looks bipartite.

One can observe that all digraphs for which we could prove in this paper that they are not avoidable allow a height function, leading us to the following problem. 
\begin{question}
Are all digraphs which allow a height function not avoidable?
\end{question}
The simplest digraph which allows a height function for which we do not know whether it is avoidable is the orientation of $C_4$ consisting of two parallel directed paths of length two. We believe this is the most difficult orientation of $C_4$, so a resolution to the above question is likely to also answer the following.
\begin{question}
Which orientations of $C_4$ are avoidable?
\end{question}

Finally, we also would like to highlight another setting. We showed that the question of avoidability is much simpler if we restrict to regular digraphs. But what if we restrict ourselves to Eulerian digraphs, i.e., digraphs where each vertex has the same out- as in-degree? We say that a digraph $F$ is \emph{Eulerian-avoidable} if there exists $d_F:\mathbb{N}\rightarrow \mathbb{N}$ such that every Eulerian digraph with minimum out-degree $d$ contains an $F$-free subdigraph with minimum out-degree $k$. Which digraphs are Eulerian-avoidable? The motivation for this question is that it might yield a new approach to gain insight into the undirected setting and Thomassen's conjecture. An interesting starting point could be to resolve the following question.

\begin{question}
Is every orientation of $C_4$ Eulerian-avoidable?
\end{question}

\paragraph*{Acknowledgments.} We would like to thank Antonio Gir\~{a}o, Freddie Illingworth and Alex Scott for interesting discussions on this subject during a visit of the last author to Oxford University in 2023. 

The last author would also like to thank Lior Gishboliner and Tibor Szab\'{o} for discussions related to the topic back in 2020.
\bibliographystyle{abbrv}
\bibliography{sources.bib}
\end{document}